\documentclass[leqno]{amsart}
\usepackage{stmaryrd,graphicx}
\usepackage{amssymb,mathrsfs,amsmath,amscd,amsthm,pxfonts}
\usepackage{float}
\usepackage[all,cmtip]{xy}
\DeclareMathAlphabet{\mathpzc}{OT1}{pzc}{m}{it}
\usepackage{amsfonts,latexsym,wasysym}

\newcommand{\ui}{[0,1]}

\newcommand{\wild}{\mathbf{w}}
\newcommand{\awild}{\mathbf{aw}}

\newcommand{\pionex}{\pi_{1}(X,x_0)}

\newcommand{\pioneh}{\pi_{1}(\bbh,b_0)}

\newcommand{\mci}{\mathcal{I}}

\newcommand{\scrc}{\mathscr{C}}

\newcommand{\scru}{\mathscr{U}}
\newcommand{\scrv}{\mathscr{V}}

\newcommand{\bbw}{\mathbb{W}}
\newcommand{\bbc}{\mathbb{C}}

\newcommand{\bbh}{\mathbb{H}}
\newcommand{\bbn}{\mathbb{N}}

\newcommand{\bbr}{\mathbb{R}}

\newcommand{\ov}{\overline}

\newtheorem{theorem}{Theorem}[section]
\newtheorem{lemma}[theorem]{Lemma}
\newtheorem{proposition}[theorem]{Proposition}
\newtheorem{corollary}[theorem]{Corollary}
\theoremstyle{definition}\newtheorem{definition}[theorem]{Definition}
\newtheorem{example}[theorem]{Example}

\newtheorem{remark}[theorem]{Remark}

\begin{document}
\title{Transfinite product reduction in fundamental groupoids}
\author{Jeremy Brazas}
\address{West Chester University\\Department of Mathematics\\ 25 University Avenue\\ West Chester, PA 19383}
\email{jbrazas@wcupa.edu}
\date{\today}

\begin{abstract}
Infinite products, indexed by countably infinite linear orders, arise naturally in the context of fundamental groupoids. Such products are called ``transfinite" if the index orders are permitted to contain a dense suborder and are called ``scattered" otherwise. In this paper, we prove several technical lemmas related to the reduction (i.e. combining of factors) of transfinite products in fundamental groupoids. Applying these results, we show that if the transfinite fundamental group operations are well-defined in a space $X$ with a scattered algebraic $1$-wild set $\mathbf{aw}(X)$, then all transfinite fundamental groupoid operations are also well-defined. 
\end{abstract}

\maketitle

\section{Introduction}

Fundamental group(oid)s of non-locally contractible spaces provide natural models of group(oid) structures enriched with non-commutative infinite product operations. The use of such infinite products is ubiquitous in the progressive literature on the homotopy and homology groups of such spaces, e.g. \cite{CChegroup,CHM,CK,E99freesubgroups,Edasingularonedim16,KentHomomorphisms}. For example, fundamental group(oid)s of one-dimensional spaces \cite{EdaSpatial,Edaonedim,EK98,FZ013caley} are very much infinitary extensions of fundamental group(oid)s of graphs since path-homotopy classes have ``reduced" representatives that are unique up to reparameterization \cite{CConedim}.

This paper seeks to develop general techniques for operating with these kinds of topologically-influenced algebraic structures. In particular, the current paper is dedicated to proving several results that involve reducing (i.e. combining factors in) such products. As in infinitary logic, we use the term ``infinitary" to refer to partially defined operations with infinitely many inputs (akin to infinite sums and products in $\bbc$) as opposed to binary, trinary, and other finitary operations appearing in standard algebraic fields. 

Prior to stating our main results, we clarify the way in which infinitary operations on paths and homotopy classes extend finite composition in ordinary categories. Suppose that $\scrc$ is a category where we choose to write the composition operation as a product operation (as in a monoid or group). By iteratively applying the face maps in the nerve of $\scrc$, we may reduce any factorization
\[\xymatrix{
X_0 \ar[r]^-{a_1} & X_1 \ar[r]^-{a_2}& X_2 \ar[r] & \cdots \ar[r] & X_{n-1} \ar[r]^-{a_{n}} & X_n
}\]
of a morphism $\alpha=a_1a_2a_3\cdots a_n$ to a ``reduced" factorization of $\alpha$:
\[\xymatrix{
X_{i_0} \ar[r]^-{b_1} & X_{i_1} \ar[r]^-{b_2}& X_{i_2} \ar[r] & \cdots \ar[r] & X_{i_{m-1}} \ar[r]^-{b_{m}} & X_{i_m}
}\]
where $0=i_0<i_1<i_2<\cdots <i_{m-1}<i_m=n$ and $b_j=a_{i_{j-1}+1}a_{i_{j-1}+2}\cdots a_{i_{j}}$. Formally, the original source/target sequence $X_0,X_1,X_2,\dots, X_n$ is a function $f:A\to Ob(\scrc)$, $i\mapsto X_i$ where $A=\{0,1,2,\dots,n\}$. The source/target sequence $X_{i_0},X_{i_1},X_{i_2},\dots, X_{i_m}$ for the reduced factorization is the restriction $f|_{B}$ where $B$ is obtained by deleting $A\cap(i_{j-1},i_j)$ for all $1\leq j\leq m$. To define an infinitary analogue of finite product reduction, we replace $\scrc$ with a topological category of paths in a space $X$ and appeal to linear order theory.

A \textit{cut-set} is a compact nowhere dense subset $A\subseteq\bbr$ and a \textit{cut-map} in a topological space $X$ is a continuous function $A\to X$ on a cut-set. Given a path $\alpha:[a,b]\to X$ and cut-set $A$ with $[\min(A),\max(A)]=[a,b]$, the restriction $\alpha|_{A}:A\to X$ is a cut-map and we refer to the pair $(\alpha,A)$ as a \textit{factorization} of $\alpha$. For a finite cut-set $A=\{a=t_0,t_1,\dots ,t_n=b\}$, the factorization $(\alpha,A)$ may be represented, uniquely up to reparameterization, by the $n$-fold concatenation $\prod_{i=1}^{n}\alpha|_{[t_{i-1},t_i]}$. If $A$ is infinite, the factorization $(\alpha,A)$ may still be represented, uniquely up to reparameterization, as an infinite concatenation of subpaths $\alpha|_{\ov{I}}$ indexed by the countably infinite linearly ordered set $\mci(A)=\{I\mid I\text{ is a component of }[a,b]\backslash A\}$. In this way, the notion of ``cut-map" is a non-discrete extension of source/target sequences in ordinary categories.

In many situations, it is helpful to represent an infinite factorization $(\alpha,A)$ using product notation: $\alpha=\prod_{I\in\mci(A)}\alpha|_{\ov{I}}$ \textit{over cut-map} $\alpha|_{A}$. However, we cannot overstate the importance of cut-maps (not represented in product notation) when dealing with infinite factorizations. For example, if $C$ is the middle third Cantor set, then there are uncountably many points $t\in C$, which are not the boundary point of an interval $I\in\mci(C)$ but which may be crucial to the structure of factorizations over a cut-map $C\to X$. Specifically, if $\alpha:\ui\to\ui$ is the standard Cantor function, then we have $\alpha=\prod_{I\in\mci(C)}\alpha|_{\ov{I}}$ over an onto cut-map $\alpha|_{C}:C\to \ui$ where each path $\alpha|_{\ov{I}}$ is constant at a dyadic rational, yet $\alpha$ is a path from $0$ to $1$. Hence, due to the role of cut-maps, we cannot jump to the tempting conclusion that a path represented in product notation as $\prod_{I\in\mci(A)}\alpha|_{\ov{I}}$ is constant if all of the factors $\alpha|_{\ov{I}}$, $I\in\mci(A)$ are constant.

A \textit{reduction} of a factorization $(\alpha,A)$ is a factorization $(\alpha,B)$ such that $B\subseteq A$. Analogous to finitary product reduction in categories, the reduced factorization $\prod_{J\in\mci(B)}\alpha|_{\ov{J}}$ is formed by combining convex collections of the factors in $\prod_{I\in\mci(A)}\alpha|_{\ov{I}}$, namely, we have $\alpha|_{\ov{J}}=\prod_{I\in\mci(A),I\subseteq J}\alpha_{\ov{I}}$ formed over cut-map $\alpha|_{A\cap\ov{J}}$. To avoid the pitfalls of infinite product notation, we will typically view $B$ as simply being obtained from $A$ by deleting $A\cap J$ for every component $J\in\mci(B)$. 

The situation for path-homotopy classes is more delicate. The primary difficulty we face is determining when reductions preserve the path-homotopy relation. Throughout, we use $\pi_1$ to refer to the fundamental group and $\Pi_1$ to refer to the fundamental groupoid. We will say that a space $X$ has
\begin{enumerate}
\item \textit{well-defined transfinite $\Pi_1$-products} if for any cut-set $A$ and paths $\alpha,\beta:\ui\to X$ such that $\alpha|_{A}=\beta|_{A}$ and $\alpha|_{\ov{I}}\simeq  \beta|_{\ov{I}}$ for all $I\in\mci(A)$, we have $\alpha\simeq\beta$.
\item \textit{well-defined transfinite $\pi_1$-products at }$x\in X$ if for any cut-set $A$ and loops $\alpha,\beta:(\ui,A)\to (X,x)$ such that $\alpha|_{\ov{I}}\simeq  \beta|_{\ov{I}}$ for all $I\in\mci(A)$, we have $\alpha\simeq\beta$. If $X$ has well-defined transfinite $\pi_1$-products at all of its points, we say $X$ has \textit{well-defined transfinite $\pi_1$-products}.
\end{enumerate}
Since every loop is a path, it is clear that (1) $\Rightarrow$ (2). If $X$ has well-defined transfinite $\Pi_1$-products (resp. well-defined transfinite $\pi_1$-products at $x\in X$), then for a fixed cut-map $A\to X$, the definition $\prod_{I\in\mci(A)}[\alpha|_{\ov{I}}]:=\left[\prod_{I\in\mci(A)} \alpha|_{\ov{I}}\right] $ is a well-defined infinitary operation on the fundamental groupoid $\Pi_1(X)$ (resp. on $\pi_1(X,x)$). However, despite the fact that group properties typically pass to their groupoid counterparts, it is shown in \cite[Theorem 1.1]{BrazDense} that there exists a locally path-connected metric space $X$ for which $X$ has well-defined transfinite $\pi_1$-products but not well-defined transfinite $\Pi_1$-products. In this paper, we bridge the gap between these infinitary group and groupoid operations to the greatest extent possible. We consider the following theorem to be the culminating result of this paper. Recall that a space $Y$ is \textit{scattered} if every non-empty subspace of $Y$ contains an isolated point. 

\begin{theorem}\label{mainthm}
Suppose $X$ is a space in which the subspace of points $\awild(X)$ at which there exists an infinite concatenation of non-contractible loops is scattered. Then  $X$ has well-defined transfinite $\pi_1$-products if and only if $X$ has well-defined transfinite $\Pi_1$-products.
\end{theorem}

We remark that nearly all of our results, including Theorem \ref{mainthm}, are completely general, applying to arbitrary topological spaces. Moreover, the example in \cite[Theorem 1.1]{BrazDense} illustrates that Theorem \ref{mainthm} is very much an optimal extension of \cite[Theorem 7.13]{BFTestMap}, which only implies Theorem \ref{mainthm} in the case that $X$ is first countable and $\awild(X)$ is discrete.

It is also shown in \cite{BrazDense} that a path-connected metrizable space $X$ admits a generalized universal covering space in the sense of \cite{FZ07} if and only if $X$ has well-defined transfinite $\Pi_1$-products. Moreover, \cite[Corollary 7.15]{BFTestMap} implies that if $\pionex$ is abelian, then the well-known homotopically Hausdorff property is also equivalent to these two properties. Combining these results with Theorem \ref{mainthm}, we obtain the following.

\begin{corollary}\label{maincor1}
Suppose $X$ is path connected, metrizable, and $\awild(X)$ is scattered. Then $X$ admits a generalized universal covering space in the sense of \cite{FZ07} if and only if $X$ has well-defined transfinite $\pi_1$-products. Moreover, if $\pionex$ is abelian, then the homotopically Hausdorff property is equivalent to these two properties.
\end{corollary}

The remainder of this paper is structured as follows. In Section \ref{wildnesssection}, we settle notation and study the difference between the set $\wild(X)$ of points at which $X$ is not semilocally simply connected and the set $\awild(X)$ of points in $X$ that cause algebraic wildness in the fundamental group of $X$. Our use of $\awild(X)$ in the place of $\wild(X)$ is precisely what allows us to prove our results for arbitrary spaces. In Section \ref{homotopycutsetsection}, we define the notion of a ``homotopy cut-set" for a pair of paths $\alpha$ and $\beta$ that agree on a cut-set $A$ and recall known well-definedness results. In Section \ref{technicalsection}, we prove several technical lemmas that identify situations where the reduction of transfinite factorizations is possible. Using the results from Section \ref{technicalsection}, we prove Theorem \ref{maintechthm} (our main technical theorem) and Theorem \ref{mainthm} in Section \ref{finalsection}.

\section{Sets of 1-dimensional wildness}\label{wildnesssection}

A \textit{path} is a continuous function $\alpha:[s,t]\to X$, which we call a \textit{loop based at} $x\in X$ if $\alpha(s)=\alpha(t)=x$. If $[a,b],[c,d]\subseteq \ui$ and $\alpha:[a,b]\to X$, $\beta:[c,d]\to X$ are maps, we write $\alpha\equiv\beta$ if $\alpha=\beta\circ \phi$ for some increasing homeomorphism $\phi: [a,b]\to [c,d]$; if $\phi$ is linear and if it does not create confusion, we will identify $\alpha$ and $\beta$. We write $\alpha\simeq\beta$ if two paths are path-homotopic. If a loop $\alpha$ is (resp. is not) path-homotopic to a constant loop, we may simply say that $\alpha$ is \textit{trivial} (resp. \textit{non-trivial}).

If $\alpha:\ui\to X$ is a path, then $\alpha^{-}(t)=\alpha(1-t)$ is the reverse path. If $\alpha_1,\alpha_2,\dots,\alpha_n$ is a sequence of paths such that $\alpha_{j}(1)=\alpha_{j+1}(0)$ for each $j$, then $\prod_{j=1}^{n}\alpha_j=\alpha_1\cdot \alpha_2\cdot\;\cdots\;\cdot \alpha_n$ is the $n$-fold \textit{concatenation} defined as $\alpha_j$ on $\left[\frac{j-1}{n},\frac{j}{n}\right]$. An infinite sequence $\alpha_1,\alpha_2,\alpha_3,\dots$ of paths in $X$ is a \textit{null sequence} if $\alpha_n(1)=\alpha_{n+1}(0)$ for all $n\in\bbn$ and there is a point $x\in X$ such that every neighborhood of $x$ contains $\alpha_n([0,1])$ for all but finitely many $n\in\bbn$. Given such a null-sequence, one may form the \textit{infinite concatenation} $\prod_{n=1}^{\infty}\alpha_n$, defined to be $\alpha_n$ on $\left[\frac{n-1}{n},\frac{n}{n+1}\right]$ and mapping $1$ to $x$.

\begin{definition}
The \textit{topological $1$-wild set} of a space $X$ is the subspace $$\wild(X)=\{x\in X\mid X\text{ is not semilocally simply connected at }x\}.$$
\end{definition}

According to classical covering space theory \cite{Spanier66}, if $X$ is path connected and locally path connected, then $\wild(X)$ is the topological obstruction to the existence of a simply connected covering space over $X$. This set plays a key role in automatic continuity and homotopy classification results for one-dimensional \cite{EdaSpatial,Edaonedim} and planar \cite{CK,KentHomomorphisms} continua since the homeomorphism type of the topological 1-wild set is a homotopy invariant within these classes of spaces \cite[Theorem 9.13]{BFThawaiianpants}.

\begin{lemma}\cite[Proposition 7.7]{BFTestMap}\label{wildisclosedlemma}
If $Y$ is locally path connected and $f:Y\to X$ is a map, then $f^{-1}(\wild(X))$ is closed in $Y$. In particular, if $X$ is locally path connected, then $\wild(X)$ is closed in $X$.
\end{lemma}

We have $x\in \wild(X)$ if there exists arbitrarily small non-trivial loops based at $x$, i.e. a net $\{\alpha_j\}_{j\in J}$ of non-trivial loops based at $x$ converging (in the compact-open topology) to the constant loop at $x$. However, the existence of ``algebraic" wildness in the fundamental group arises from the existence of a null sequence $\{\alpha_n\}_{n\in\bbn}$ of non-trivial loops based at a point so that one can form infinite concatenations such as $\prod_{n=1}^{\infty}\alpha_n$. Examples \ref{example1} and \ref{example2} below illustrate that these two concepts do not always coincide.

\begin{definition}
We say $\pi_1(X,x_0)$ \textit{is infinitary at} $x\in X$ if there exists a null-sequence $\{\alpha_n\}$ of non-trivial loops in $X$ based at $x$. We say $\pi_1(X,x_0)$ is \textit{finitary at $x\in X$} if $X$ is not infinitary at $x$.  If there exists a point at which $\pionex$ is infinitary, we say $\pi_1(X,x_0)$ is \textit{infinitary}. Otherwise, we say $\pionex$ is \textit{finitary}.
\end{definition}

Observe that the property of a fundamental group $\pi_1(X,x_0)$ being infinitary at a point $x$ is not just a property of the abstract group $\pi_1(X,x_0)$ but rather a property of the infinitary product structure of $\pi_1(X,x)$ that is embedded in $\pionex$ via path-conjugation. In particular, this property detects whether or not the local topology at $x$ affects the algebra of $\pi_1(X,x_0)$. 

\begin{definition}
The \textit{algebraic $1$-wild set} of a space $X$ is the subspace $$\awild(X)=\{x\in X\mid \pionex\text{ is infinitary at }x\}.$$
\end{definition}

Let $C_n=\{(x,y)\in\bbr^2\mid (x-1/n)^2+y^2=(1/n)^2\}$, $n\in\bbn$ and $\bbh=\bigcup_{n\in\bbn}C_n$ be the usual Hawaiian earring space with basepoint $b_0=(0,0)$. Consider the loop $\ell_n(t)=\left(\frac{1}{n}(\cos(2\pi t-\pi)+1),\frac{1}{n}\sin(2\pi t-\pi)\right)$ traversing $C_n$ in the counterclockwise direction. Clearly, $\awild(\bbh)=\wild(\bbh)=\{b_0\}$ since $\{\ell_n\}$ is a null-sequence of loops based at $b_0$ and $\bbh$ is locally contractible at all other points. The following proposition is a direct consequence of the fact that maps $f:(\bbh,b_0)\to (X,x)$ are in bijection correspondence with null sequences $\{\alpha_n\}$ of loops based at $x$ (where $\alpha_n=f\circ\ell_n$).

\begin{proposition}
We have $x\in\awild(X)$ if and only if there exists a map $f:(\bbh,b_0)\to (X,x)$ such that $f\circ\ell_n$ is non-trivial for all $n\in\bbn$.
\end{proposition}

Let $\bbh_{\geq n}=\bigcup_{k\geq n}C_k$ denote the smaller copies of $\bbh$. Since $\bbh_{\geq n}$ is a retract of $\bbh$, the group $\pi_1(\bbh_{\geq n},b_0)$ may be identified as a subgroup of $\pioneh$. If there exists a map $f:(\bbh,b_0)\to (X,x)$, natural numbers $m_1<m_2<m_3<\cdots$, and loops $\gamma_n:(\ui,\{0,1\})\to(\bbh_{m_n},b_0)$ such that $f\circ\gamma_n$ is non-trivial for all $n\in\bbn$, then the map $g:\bbh\to X$ defined by $g\circ\ell_n=f\circ\gamma_n$ witnesses the fact that $\pionex$ is infinitary at $x$. The next proposition is a direct consequence of this observation.

\begin{proposition}\label{finitarychar}
$\pionex$ is finitary at $x$ if and only if for every map $f:(\bbh,b_0)\to (X,x)$, there exists an $m\in\bbn$ such that the induced homomorphism $(f|_{\bbh_{\geq m}})_{\#}:\pi_1(\bbh_{\geq m},b_0)\to\pi_1(X,x)$ is trivial.
\end{proposition}

\begin{remark}
The characterization in Proposition \ref{finitarychar} suggests that the ``finitary" property is a topological version of the purely group theoretic ``noncommutatively slender" property introduced by Eda \cite{Edafreesigmaproducts}: A group $G$ is \textit{noncommutatively slender} if for every homomorphism $h:\pi_1(\bbh,b_0)\to G$, there exists an $m\in\bbn$ such that $h(\pi_1(\bbh_{\geq m},b_0))=\{1\}$. Clearly, if $\pionex$ is noncommutatively slender, then $\pionex$ is finitary.

On the other hand, we note that all noncommutatively slender groups are torsion free. For any finite group $G\neq \{1\}$, the set $Hom(\pi_1(\bbh,b_0),G)$ of homomorphisms $\pi_1(\bbh,b_0)\to G$ is known to be uncountable \cite{ConnerSpencer}. However, $\pi_1(\bigcup_{1\leq k\leq n}\bbh_k,b_0)=F_n$ is free on $n$-generators and there are only finitely many homomorphism $F_n\to G$ for given $n\in\bbn$. Therefore, since $\pi_1(\bbh,b_0)\cong \pi_1(\bbh_{\geq n+1},b_0)\ast F_n$ for all $n\in\bbn$, there must exist a homomorphism $f:\pi_1(\bbh,b_0)\to G$ such that $f(\pi_1(\bbh_{\geq n},b_0))\neq \{1\}$. We conclude that if $X$ is any semilocally simply connected space and $\pionex$ has non-trivial torsion, then $\pionex$ is finitary but fails to be noncommutatively slender.
\end{remark}

\begin{proposition}
For any space $X$, we have $\awild(X)\subseteq \wild(X)$. If $X$ is first countable, then $\awild(X)= \wild(X)$.
\end{proposition}

\begin{proof}
If $x\in \awild(X)$, there exists a map $f:(\bbh,b_0)\to (X,x)$ such that $f\circ\ell_n$ is non-trivial for all $n\in\bbn$. Let $U$ be a neighborhood of $x$. Using the continuity of $f$, we can find $m\in\bbn$ such that $f(\bbh_{\geq m})\subseteq U$. In particular $f\circ\ell_m$ has image in $U$ but is non-trivial in $X$. Since every neighborhood of $x$ contains a loop that is non-trivial in $X$, we have $x\in\wild(X)$. For the second statement, suppose $X$ is first countable and $x\in\wild(X)$. Consider a countable neighborhood base $U_1\supseteq U_2\supseteq U_3\supseteq \cdots$ at $x$. By assumption, there exists a loop $\gamma_n$ in $U_n$ based at $x$ that is non-trivial in $X$. The map $f:\bbh\to X$ defined by $f\circ\ell_n=\gamma_n$ is continuous (since the basis $\{U_n\}$ is nested) and witnesses the fact that $x\in\awild(X)$.
\end{proof}

\begin{example}\label{example1}
Let $X$ be the reduced suspension of the first uncountable compact ordinal $\omega_1+1=\omega_1\cup\{\omega_1\}$ with basepoint $\omega_1$. We take $x_0$ to be the canonical basepoint of $X$. The underlying set of $X$ may be identified with a one-point union of circles $\bigvee_{\alpha\in\omega_1}C_{\alpha}$ indexed by the set of countable ordinals. Every neighborhood of $x_0$ contains $\bigcup_{\beta\leq\alpha<\omega_1}C_{\alpha}$ for some $\beta\in\omega_1$ and each circle $C_{\alpha}$ is a retract of $X$. Since $X$ is semilocally simply connected at all points in $X\backslash \{x_0\}$, we conclude that $\wild(X)=\{x_0\}$. However, using the fact that no sequence of countable ordinals converges to $\omega_1$, one can show that no path in $X$ maps onto $\bigvee_{\alpha\in S}C_{\alpha}$ for an infinite subset $S\subset\omega_1$. It follows that $\pionex$ is finitary, i.e. $\awild(X)=\emptyset$. Moreover, $\pi_1(X,x_0)$ is isomorphic to the free group $F(\omega_1)$ on uncountably many generators since the above argument extends to show that the canonical continuous bijection $\bigvee_{\alpha\in\omega_1}S^1\to X$ from the wedge of circles with the CW-topology is a weak homotopy equivalence. 
\end{example}

\begin{figure}[H]
\centering \includegraphics[height=2in]{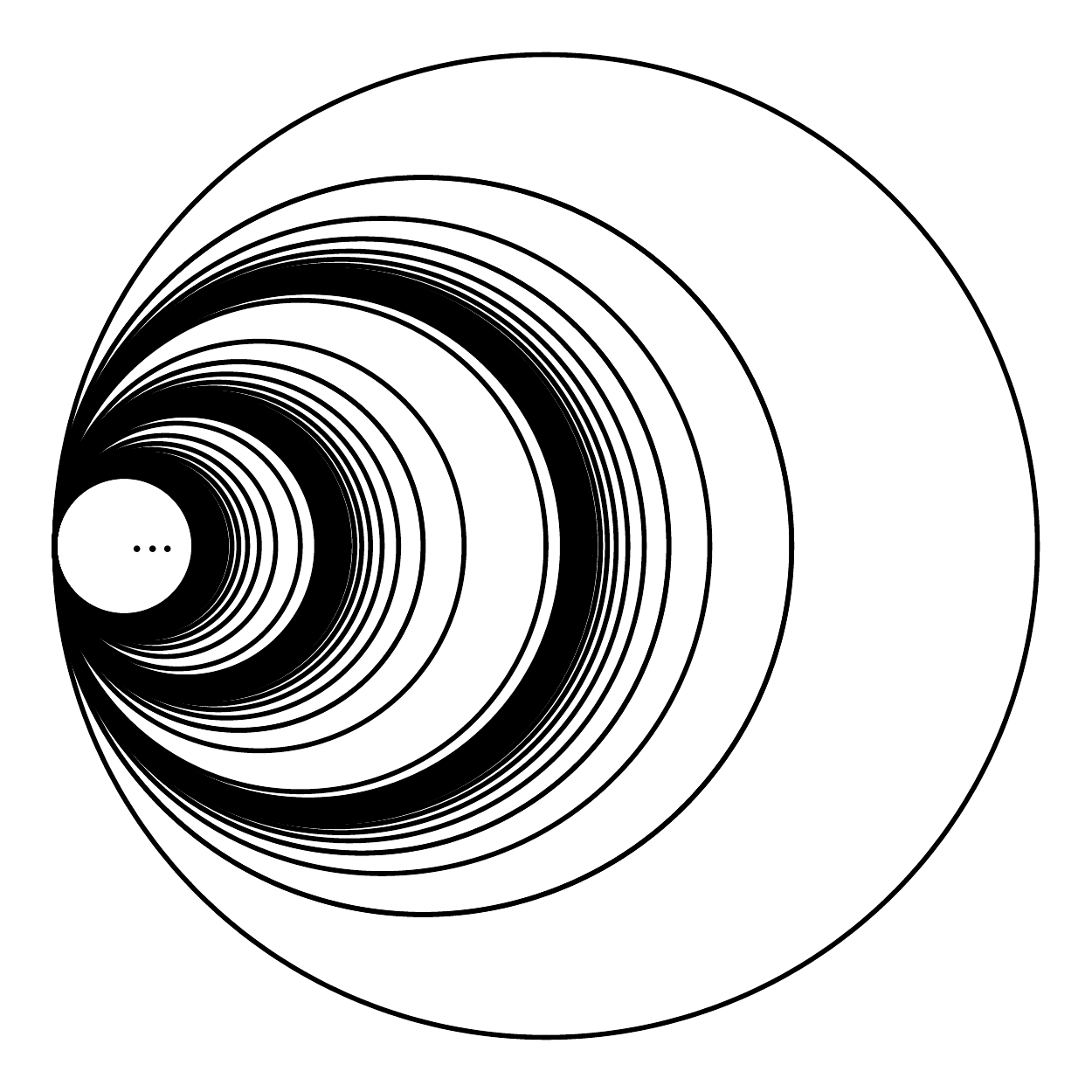}
\caption{The subspace of $X$ in Example \ref{example1} including the circles corresponding to $\alpha<\omega+\omega+\omega$.}
\end{figure}

\begin{example}\label{example2}
Let $H_n$, $n\in\bbn$ be a homeomorphic copy of $\bbh$ and let $X$ be the space obtained by attaching the family $\{H_n\}$ to the unit interval $\ui$ by identifying the basepoint of $H_n$ with $1/n\in\ui$. We give $X$ the weak topology with respect to the subspaces $\{\ui,H_1,H_2,H_3,\dots\}$. This space is locally path connected and compactly generated but is not first countable at $0$. In fact, since $X$ is the quotient of the topological sum $\ui\sqcup \coprod_{n}H_n$ of locally path-connected metric spaces, $X$ is a \textit{$\Delta$-generated space} in the sense that $U\subseteq X$ is open if and only if $\alpha^{-1}(U)$ is open in $\ui$ for all paths $\alpha:\ui\to X$; See \cite{CSWDiffeology,FRdirectedhomotopy}.

Although it is clear that $\wild(X)=\{0,\dots,1/3,1/2,1\}$, every compact subset of $X$ must lie in a subspace $X_m=\ui\cup H_1\cup H_2\cup\,\cdots\,\cup H_m$, $m\in\bbn$, which is locally contractible at $0$. Hence, for every map $f:(\bbh,b_0)\to (X,0)$, there exists $n\in\bbn$ such that $f(\bbh_{\geq n})\subseteq [0,1/m)$. Therefore, $\pi_1(X,0)$ is finitary at $0$ and, moreover, we have $\awild(X)=\{\dots,1/3,1/2,1\}$. This example shows that even after restricting to some reasonable categories of spaces, $\awild(X)$ need not be sequentially closed.
\end{example}

\begin{figure}[H]
\centering \includegraphics[height=1.3in]{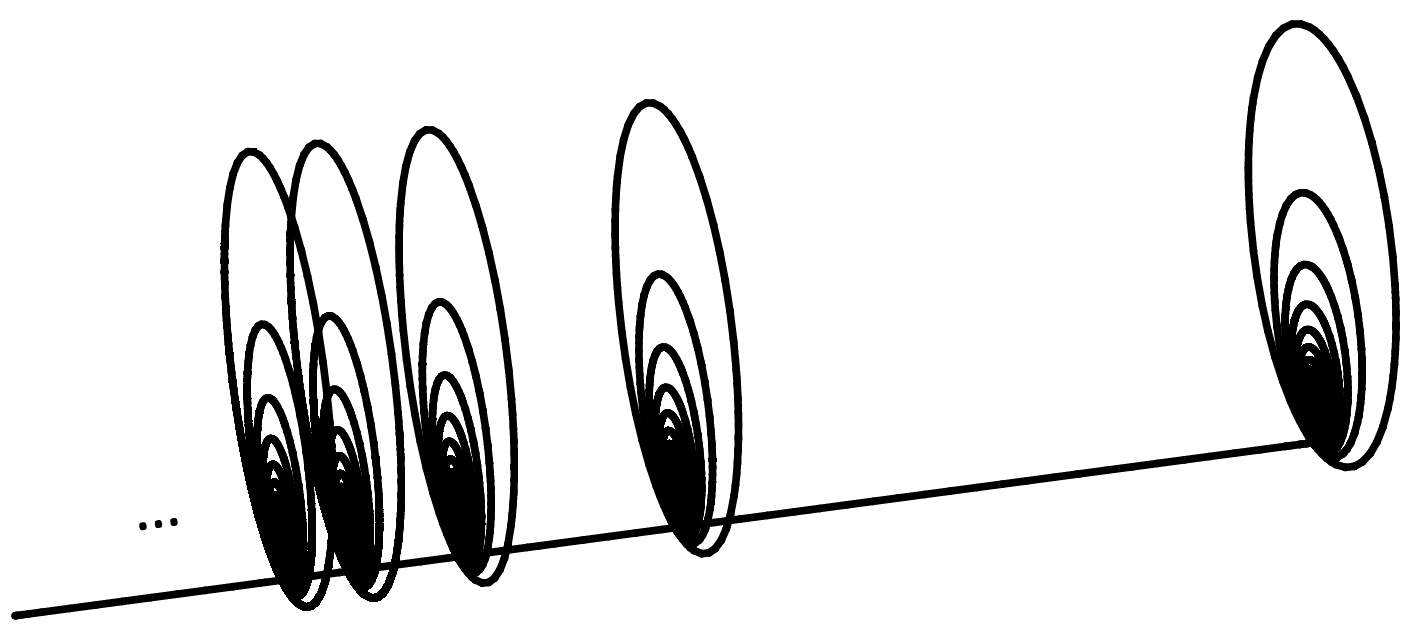}
\caption{The space $X$ in Example \ref{example2}}
\end{figure}

Within the class of spaces $X$ with a certain rigidity property (called the discrete monodromy property in \cite[Section 9]{BFThawaiianpants}), the homeomorphism type of $\wild(X)$ is a homotopy invariant of $X$. However, this is clearly not true in general since $\awild(X)$ and $\awild(X\times\ui)=\awild(X)\times \ui$ need not be homeomorphic. In the next theorem, we show that the homotopy types of the topological 1-wild set and the algebraic $1$-wild set are both invariants of homotopy type.

\begin{theorem}
The homotopy types of $\wild(X)$ and $\awild(X)$ are homotopy invariants of $X$.
\end{theorem}

\begin{proof}
We prove the result for $\awild(X)$. The proof for $\wild(X)$ simply requires replacing the sequences of non-trivial loops described using maps $\bbh\to X$ with nets of non-trivial loops. We begin by making two general observations. First, notice that if $f:(X,x)\to (Y,f(x))$ induces an injection $f_{\#}$ on fundamental groups for all $x\in X$, then $f(\awild(X))\subseteq \awild(Y)$. Secondly, suppose $H:X\times \ui\to X$ is a homotopy such that $H(x,0)=x$. We claim that $H(\awild(X)\times \ui)\subseteq \awild(X)$. If $(a,t)\in\awild(X)\times\ui$, then there is a map $k:(\bbh,b_0)\to (X,a)$ such that $k\circ\ell_n$ is non-trivial for each $n\in\bbn$. Let $H_t(x)=H(x,t)$. Since $H_0\circ k=k$ is freely homotopic to $H_t\circ k:(\bbh,b_0)\to (X,H(a,t))$, the loop $H_t\circ k\circ\ell_n$ is non-trivial for each $n\in\bbn$. Therefore, $H(a,t)\in \awild(X)$.

Suppose $f:X\to Y$ and $g:Y\to X$ are homotopy inverses. By the first observation, we have $g(f(\awild(X)))\subseteq g(\awild(Y))\subseteq \awild(X)$ and similarly, $f(g(\awild(Y)))\subseteq \awild(Y)$. We check that $f|_{\awild(X)}$ and $g|_{\awild(Y)}$ are homotopy inverses. Let $H:X\times \ui\to X$ be a homotopy from $H(x,0)=x$ to $H(x,1)=g(f(x))$. By the second observation, $H$ restricts to a homotopy $G=H|_{\awild(X)\times\ui}:\awild(X)\times\ui\to\awild(X)$ where $G(x,0)=x$ and $G(x,1)=g|_{\awild(Y)}(f|_{\awild(X)}(x))$. Thus $g|_{\awild(Y)}\circ f|_{\awild(X)}\simeq id_{\awild(X)}$. The symmetric argument shows that $f|_{\awild(X)}\circ g|_{\awild(Y)}\simeq id_{\awild(Y)}$. 
\end{proof}

\begin{example}
Recall that the space $X$ constructed in Example \ref{example2} has a discrete algebraic $1$-wild set $\awild(X)=\{\dots,1/3,1/2,1\}$ and consider the subspace $X'=\{(t,0,0)\in\bbr^3\mid t\in\ui\}\cup\bigcup_{n\in\bbn}\{1/n\}\times \bbh$ of $\bbr^3$. There is a continuous bijection $X\to X'$; however, since $X'$ is a metric space $\awild(X')=\wild(X')=\{0,\dots,1/3,1/2,1\}$. Since $\awild(X)$ and $\awild(X')$ are totally path-disconnected but not homeomorphic, they are not homotopy equivalent. Therefore, $X$ and $X'$ are not homotopy equivalent.
\end{example}

More generally, if $\wild(X)\nsimeq\awild(X)$ for a given space $X$, then $X$ is not homotopy equivalent to any first countable space.

\section{Homotopy Cut-Sets}\label{homotopycutsetsection}

\begin{definition}
If $A$ is a non-degenerate compact subset of $\bbr$, let $\mci(A)$ denote the set of components of $[\min(A),\max(A)]\backslash A$ equipped with the linear ordering inherited from $\bbr$. 
\end{definition}
Note that $\mci(A)$ is always countable and if $A$ is nowhere dense, then $A$ may be identified with the set of cuts of $\mci(A)$. We refer to \cite{RosensteinLO} for basic linear order theory.

\begin{definition}
Given two paths $\alpha,\beta:[s,t] \to X$, a set $A\subseteq [s,t]$ is called a \textit{homotopy cut-set for} $\alpha$ \textit{and} $\beta$ if 
\begin{enumerate}
\item $A$ is closed, nowhere dense, and contains $\{s,t\}$,
\item $\alpha|_{A}=\beta|_{A}$,
\item for every component $(a,b)\in\mci(A)$, $\alpha|_{[a,b]}$ and $ \beta|_{[a,b]}$ are path-homotopic.
\end{enumerate}
\end{definition}

\begin{remark}
As noted in \cite[Remark 7.2]{BFTestMap}, there is no harm in insisting that homotopy cut-sets are nowhere dense. If we allowed a homotopy cut-set $A$ for paths $\alpha$ and $\beta$ to have a non-empty interior, then the assumption $\alpha|_{A}=\beta|_{A}$ ensures that for any non-empty component $(a,b)$ of $int(A)$, we have $\alpha|_{[a,b]}=\beta|_{[a,b]}$. Hence, we could replace $A$ with its boundary, which is nowhere dense.
\end{remark}

If there is a finite homotopy cut-set $A=\{a_0,a_1,a_2,\dots,a_n\}$ for paths $\alpha,\beta:\ui\to X$, then $\alpha\equiv \prod_{i=1}^{n}\alpha_i$ and $\beta\equiv \prod_{i=1}^{n}\beta_i$ where $\alpha_i\simeq\beta_i$ for all $i$. Finite concatenation of path-homotopies gives $\alpha\simeq\beta$. More generally, if $A$ is any homotopy cut-set for $\alpha$ and $\beta$ and $F$ is a finite subset of isolated points in $A$, then $(A\backslash F)\cup \{0,1\}$ is a homotopy cut-set for $\alpha$ and $\beta$. 

\begin{remark}
If $A$ is a homotopy cut-set for $\alpha,\beta:[s,t]\to X$ and if we have $c,d\in A$ with $c<d$, then $A\cap [c,d]$ is a homotopy cut-set for $\alpha|_{[c,d]}$ and $\beta|_{[c,d]}$.
\end{remark}

\begin{definition}\label{fourdefs}
A space $X$ has
\begin{enumerate}
\item \textit{well-defined scattered $\pi_1$-products} if any pair of loops (based at the same point) in $X$ that admit a scattered homotopy cut-set are path-homotopic.
\item \textit{well-defined scattered $\Pi_1$-products} if any pair of paths $\ui\to X$ that admit a scattered homotopy cut-set are path-homotopic.
\item \textit{well-defined transfinite $\pi_1$-products} if any pair of loops (based at the same point) in $X$ that admit a homotopy cut-set are path-homotopic.
\item \textit{well-defined transfinite $\Pi_1$-products} if any pair of paths $\ui\to X$ that admit a homotopy cut-set are path-homotopic.
\end{enumerate}
\end{definition}

Clearly, we have (4) $\Rightarrow$ (3) $\Rightarrow$ (1) and (4) $\Rightarrow$ (2) $\Rightarrow$ (1). Less trivially, the converse (1) $\Rightarrow$ (2) holds in general; See Theorem \ref{scatteredproducthm} below. The converse of (4) $\Rightarrow$ (3) does not hold for all locally path-connected metrizable spaces \cite[Theorem 1.1]{BrazDense}. It is an open question if the converse of (3) $\Rightarrow$ (1) holds. 

\begin{remark}\label{homotopyinvarianceremark}
All four of the properties defined in Definition \ref{fourdefs} are homotopical properties, i.e. invariant under homotopy type. This observation can be proved directly; however, it also follows from the more general test-map framework developed in \cite{BFTestMap}; See Remark 2.11 in \cite{BFTestMap}.
\end{remark}

Recall that a space $X$ is \textit{homotopically Hausdorff} if there does not exist a non-trivial based loop $(\ui,\{0,1\})\to (X,x)$ that is path-homotopic to some loop in every neighborhood of $x$. We refer to \cite{BFTestMap,CMRZZ08,FRVZ11} for characterizations and examples related to this property.

\begin{theorem}\cite{BrazScattered}\label{scatteredproducthm}
For any path-connected space $X$, the following are equivalent:
\begin{enumerate}
\item any pair of loops in $X$ admitting a homotopy cut-set $A$ such that $\mci(A)$ has order type $\omega$ are path-homotopic,
\item any pair of paths in $X$ admitting a homotopy cut-set $A$ such that $\mci(A)$ has order type $\omega$ are path-homotopic,
\item $X$ has well-defined scattered $\pi_1$-products,
\item $X$ has well-defined scattered $\Pi_1$-products.
\end{enumerate}
Moreover, these four conditions are implied by the homotopically Hausdorff property. If $X$ is first countable, then these four conditions are equivalent to the homotopically Hausdorff property.
\end{theorem}

\section{Reductions of Homotopy Cut-sets}\label{technicalsection}

In this section, we work carefully to determine when the reduction $(\alpha,B)$ of a factorization $(\alpha,A)$ passes to a well-defined reduction on homotopy classes.

\subsection{Reduction at non-wild cut points}

In algebraic topology, it is standard to use a Lebesgue number to subdivide a path so that every segment of the subdivision is mapped into some neighborhood of a given open cover. The following lemma is a variant of this technique more suited to our purposes. If $A\subseteq \ui$, we say two points $a$ and $b$ are \textit{consecutive in} $A$ if $a,b\in A$, $a<b$, and $A\cap (a,b)=\emptyset$.

\begin{lemma}\label{subdivisionlemma}
Suppose $A\subseteq \ui$ is a cut-set, $\alpha,\beta:\ui\to X$ are paths such that $\alpha|_{A}=\beta|_{A}$, and $\scru$ is a cover of $\alpha(A)=\beta(A)$ by open sets in $X$. Then there exists a subdivision $0=t_0<t_1<t_2<\cdots <t_m=1$ of points in $A$ such that, for each $j\in \{1,2,\dots ,m\}$, either $t_{j-1}$ and $t_j$ are consecutive points in $A$ or $\alpha([t_{j-1},t_j])\cup \beta([t_{j-1},t_j])\subseteq U$ for some $U\in\scru$.
\end{lemma}

\begin{proof}
Since $A$ is a compact metric space, we may find $n\in\bbn$ such that if $a,b\in A$ and $|a-b|\leq 1/n$, then $\alpha(A\cap [a,b])$ lies in some $U\in\scru$. We build the desired subdivision $S=\{t_0,t_1,\dots,t_m\}$ by defining $S_k=S\cap \left[\frac{k-1}{n},\frac{k}{n}\right]$ for each $k\in\{1,2,\dots ,n\}$. If $k$ is such that $A\cap\left[\frac{k-1}{n},\frac{k}{n}\right]$ is finite (possibly empty), we set $S_k=A\cap\left[\frac{k-1}{n},\frac{k}{n}\right]$. On the other hand, fix a $k$ such that $A\cap\left[\frac{k-1}{n},\frac{k}{n}\right]$ is infinite. Let $p$ and $q$ be the minimal and maximal elements of $A\cap\left[\frac{k-1}{n},\frac{k}{n}\right]$ respectively and fix $U\in\scru$ such that $\alpha\left(A\cap\left[\frac{k-1}{n},\frac{k}{n}\right]\right)\subseteq U$. Consider the infinite set of components $\{I_1,I_2,I_3,\dots\}$ of $[p,q]\backslash (A\cap[p,q])$. Due to the compactness of $[p,q]$ and the continuity of $\alpha|_{[p,q]}$ and $\beta|_{[p,q]}$, there can only be finitely many $v\in\bbn$ such that $\alpha(I_v)\cap X\backslash U\neq \emptyset$ or $\beta(I_v)\cap X\backslash U\neq \emptyset$. Therefore, by taking the union of $\{p,q\}$ and the boundaries of these finitely many open intervals $I_v$, we may find a subdivision $p=s_0<s_1<s_2<\dots<s_r=q$ using points in $A$ such that, for each $i\in \{1,2,\dots, r\}$, either $s_{i-1}$ and $s_{i}$ are consecutive in $A\cap [p,q]$ or $\alpha([s_{i-1},s_i])\cup \beta([s_{i-1},s_i])\subseteq U$. Set $S_k=\{s_0,s_1,s_2,\dots,s_r\}$. Since $\{0,1\}$ has been automatically included in $S$, the construction of $S$ makes it clear that this subdivision sufficiently satisfies the desired conditions.
\end{proof}

The author finds it interesting that the following proof is essentially identical in structure to the standard proof of the compactness of $\ui$.

\begin{lemma}\label{nonwildlemma1}
If $A$ is a homotopy cut-set for paths $\alpha,\beta:\ui\to X$, and $\alpha(A)=\beta(A)\subseteq X\backslash\awild(X)$, then $\alpha\simeq\beta$.
\end{lemma}

\begin{proof}
We begin the proof by fixing $t\in A$, setting $x=\alpha(t)$, and considering the four ways in which $t$ may be surrounded by other points of $A$: 
\begin{enumerate}
\item If there exists $d\in A$ such that $t$ and $d$ are consecutive in $A$, then $\alpha|_{[t,d]}\simeq \beta|_{[t,d]}$ since $A$ is a homotopy cut-set.
\item If there exists $c\in A$ such that $c$ and $t$ are consecutive in $A$, then $\alpha|_{[c,t]}\simeq \beta|_{[c,t]}$ since $A$ is a homotopy cut-set.
\item If there exists $t<\cdots <t_3<t_2<t_1$ in $A$ converging to $t$, then we may define a map $f:(\bbh,b_0)\to (X,x)$ by $f\circ\ell_n=\alpha|_{[t,t_n]}\cdot\beta|_{[t,t_n]}^{-}$; continuity of $f$ follows from the continuity of $\alpha$ and $\beta$ at $t$. Since $x\notin \awild(X)$, there exists $N$ such that $\alpha|_{[t,t_N]}\cdot\beta|_{[t,t_N]}^{-}$ is null-homotopic, i.e. $\alpha|_{[t,t_N]}\simeq\beta|_{[t,t_N]}$.
\item If there exists $t_1<t_2<t_3<\cdots <t$ in $A$ converging to $t$, then we map apply the same idea as in (3) to see that $\alpha|_{[t_N,t]}\simeq\beta|_{[t_N,t]}$ for some $N$.
\end{enumerate}

Let $U=\{t\in A\mid \alpha|_{[0,t]}\simeq\beta|_{[0,t]}\}$ and consider $u=\sup(U)$. Either there is a point $t\in A$ such that $0$ and $t$ are consecutive in $A$ or there exists $\cdots t_3<t_2<t_1$ in $A$ converging to $0$. Applying Situations (1) and (3) above respectively, we see that there exists $v>0$ in $A$ such that $\alpha|_{[0,v]}\simeq \beta|_{[0,v]}$. Thus $u>0$.

Either $u\in U$ or there exists $t_1<t_2<t_3<\cdots<u$ in $U$ converging to $u$. In the latter case, we have $\alpha|_{[0,t_n]}\simeq \beta|_{[0,t_n]}$ for all $n\in\bbn$. Applying Situation (4) above, we see that $\alpha|_{[t_N,u]}\simeq\beta|_{[t_N,u]}$ for some $N$. Concatenation of path-homotopies shows that $\alpha|_{[0,u]}\simeq\beta|_{[0,u]}$. Hence, $u\in U$.

Finally, suppose $u<1$. There cannot be $u<t$ such that $u,t$ are consecutive in $A$ for Situation (1) and concatenation would imply that $u$ is a not an upper bound for $U$. Therefore, there must exist a sequence $u<\cdots <t_3<t_2<t_1$ in $A$ converging to $u$. However, Situation (3) implies that $\alpha|_{[u,t_N]}\simeq \beta|_{[u,t_N]}$ for some $N$. Concatenating with $\alpha|_{[0,u]}\simeq\beta|_{[0,u]}$ gives $\alpha|_{[0,t_N]}\simeq\beta|_{[0,t_N]}$; a contradiction that $u$ is an upper bound for $U$. Thus $u=1$.
\end{proof}

We also require the following lemma, which allows us to deal with endpoints of paths.

\begin{lemma}\label{isolatedendpoints}
Suppose $\alpha,\beta:\ui\to X$ are paths from $x_0$ to $x_1$ and $A$ is a homotopy cut-set for $\alpha$ and $\beta$.
\begin{enumerate}
\item If $x_0\notin \awild(X)$, then there exists a homotopy cut-set $B\subseteq A$ for $\alpha$ and $\beta$ such that $0$ is an isolated point of $B$.
\item If $x_1\notin \awild(X)$, then there exists a homotopy cut-set $B\subseteq A$ for $\alpha$ and $\beta$ such that $1$ is an isolated point of $B$.
\item If $x_0,x_1\notin \awild(X)$, then there exists a homotopy cut-set $B\subseteq A$ for $\alpha$ and $\beta$ such that $0$ and $1$ are isolated points of $B$.
\end{enumerate}
\end{lemma}

\begin{proof}
We prove (1). The proof of (2) is symmetric and (3) follows by sequentially applying (1) and (2). If $0$ is already an isolated point of $A$, take $B=A$. Otherwise, there exists $\cdots <t_3<t_2<t_1$ in $A$ converging to $0$. Applying Situation (3) in Lemma \ref{nonwildlemma1} to $t=0$, there must be an $N\in\bbn$ such that $\alpha|_{[0,t_N]}\simeq \beta|_{[0,t_N]}$. The set $B=\{0\}\cup (A\cap [t_N,1])$ is the desired homotopy cut-set for $\alpha$ and $\beta$.
\end{proof}

\begin{lemma}\label{scatteruplemma}
Suppose $X$ has well-defined scattered $\Pi_1$-products. If $A$ is a homotopy cut-set for paths $\alpha,\beta:\ui\to X$ such that for all $a,b\in A\cap (0,1)$ with $a<b$, we have $\alpha|_{[a,b]}\simeq\beta|_{[a,b]}$, then $\alpha\simeq\beta$.
\end{lemma}

\begin{proof}
We define a scattered homotopy cut-set $S\subseteq A$ for $\alpha$ and $\beta$ as follows. We consider four cases:
\begin{enumerate}
\item If $0$ and $1$ are isolated points in $A$, find $0<s<t<1$ in $A$ such that $0,s$ and $t,1$ are consecutive in $A$. Let $S=\{0,s,t,1\}$.
\item If $0$ is isolated in $A$ but $1$ is not, find $0<s<t_1<t_2<t_3<\cdots <1$ in $A$ with $0,s$ consecutive in $A$ and $\{t_n\}\to 1$. Let $S=\{0,s,t_1,t_2,t_3\dots,1\}$.
\item If $1$ is isolated in $A$ but $0$ is not, find $0<\cdots s_3<s_2<s_1<t<1$ in $A$ with $t,1$ consecutive in $A$ and $\{s_n\}\to 0$. Let $S=\{0,\dots,s_3,s_2,s_1,t,1\}$.
\item If neither $0$ nor $1$ is isolated in $A$, find $0<\cdots s_3<s_2<s_1<t_1<t_2<t_3<\cdots <1$ in $A$ with $\{s_n\}\to 0$ and $\{t_n\}\to 1$. Let $S=\{0,\dots,s_3,s_2,s_1,t_1,t_2,\dots,1\}$.
\end{enumerate}
In any of the four cases, $S$ is a scattered cut-set and, by assumption, is a homotopy cut-set for $\alpha$ and $\beta$. Since $X$ is assumed to have well-defined scattered $\Pi_1$-products, we have $\alpha\simeq\beta$.
\end{proof}

\begin{lemma}\label{awildsetreductiongoodlemma}
Suppose $X$ has well-defined scattered $\Pi_1$-products. If $A$ is a homotopy cut-set for paths $\alpha,\beta:\ui\to X$, then there exists a homotopy cut-set $B\subseteq A$ for $\alpha$ and $\beta$ such that $\alpha(B\cap (0,1))=\beta(B\cap (0,1))\subseteq \awild(X)$.
\end{lemma}

\begin{proof}
Let $\ui_{\alpha}$ and $\ui_{\beta}$ be disjoint copies of the unit interval and let $\bbw_{A}=\ui_{\alpha}\cup\ui_{\beta}/\mathord{\sim}$ where we identify $t_{\alpha}\mathord{\sim} t_{\beta}$ if $t\in A$. When $t\in A$ we simply write $t$ for the image of $\{t_{\alpha},t_{\beta}\}$ in $\bbw_A$. Note that $\bbw_{A}$ is a one-dimensional Peano continuum homeomorphic to the planar construction given in \cite{BrazScattered}. Let $p_{\alpha}:\ui\to \bbw_{A}$ and $p_{\beta}:\ui\to \bbw_{A}$ be the paths, which are the restriction of the quotient map to $\ui_{\alpha}$ and $\ui_{\beta}$ respectively. Together, the paths $\alpha,\beta$ uniquely induce a map $f:\bbw_{A}\to X$ given by $f\circ p_{\alpha}=\alpha$ and $f\circ p_{\beta}=\beta$. 

Let $U$ be the set of $t\in A\cap (0,1)$ such that there exists an $\epsilon>0$ so the path-connected basic open set $N(t,\epsilon)=p_{\alpha}((t-\epsilon,t+\epsilon))\cup p_{\beta}((t-\epsilon,t+\epsilon))$ is contained in $\bbw_A\backslash \{0,1\}$ and contains no path $\gamma$ such that $f\circ\gamma$ is a non-trivial loop in $X$. Clearly, $U$ is open in $A$. Hence, $B=A\backslash U$ is closed in $A$ and thus closed and nowhere dense in $\ui$.

First, we check that $\alpha(B\cap (0,1))\subseteq \awild(X)$. Let $t\in B\cap (0,1)$. Recall that $\{N(t,1/n)\}_{n\in\bbn}$ is a neighborhood base of path-connected open sets at $t$. Then there exists a path $\gamma_n:\ui\to N(t,1/n)$ such that $f\circ\gamma_n$ is a non-trivial loop in $X$. Find a path $\delta_n:\ui\to N(t,1/n)$ from $t$ to $\gamma_{n}(0)$. Define $g:(\bbh,b_0)\to (X,f(t))$ by $g\circ\ell_n=(f\circ\delta_n)\cdot (f\circ\gamma_n)\cdot(f\circ\delta_n)^{-}$. The continuity of $g$ at $b_0$ follows directly from the continuity of $f$ and fact that the basis at $t$ is nested. Moreover, $g\circ\ell_n$ cannot be null-homotopic for any $n$ since $f\circ\gamma_n$ is not. We conclude that $\alpha(t)=f(t)\in\awild(X)$.

Next, we use Lemma \ref{scatteruplemma} to show that $B$ is a homotopy cut-set for $\alpha$ and $\beta$. Given $(c,d)\in\mci(B)$, we must show that $\alpha|_{[c,d]}\simeq\beta|_{[c,d]}$. Consider any two elements $a<b$ in $A\cap (c,d)$. By Lemma \ref{scatteruplemma}, it suffices to show $\alpha|_{[a,b]}\simeq\beta|_{[a,b]}$. Note that if $t\in A\cap [a,b]$, then $t\in A\backslash B=U$ and there must exist a basic open neighborhood $N(t,\epsilon_t)$ of $t$ in $\bbw_A$ such that $f$ maps no path in $N(t,\epsilon_t)$ to a non-trivial loop in $X$. Then $\scrv=\{N(t,\epsilon_t)\mid t\in A\cap [a,b]\}$ is an open cover of $(p_{\alpha})(A\cap [a,b])=(p_{\beta})(A\cap [a,b])$ by open sets in $\bbw_A$ such that $f$ maps no path in an element of $\scrv$ to a non-trivial loop in $X$. Applying Lemma \ref{subdivisionlemma} to the paths $p_{\alpha}$, $p_{\beta}$ and cut-set $A\cap [a,b]$, we may find a subdivision $a=t_0<t_1<t_2<\cdots <t_m=b$ of points in $A$ such that, for each $j\in \{1,2,\dots ,m\}$, either $t_{j-1}$ and $t_j$ are consecutive points in $A$ or $p_{\alpha}([t_{j-1},t_j])\cup p_{\beta}([t_{j-1},t_j])$ lie in some element of $\scrv$. If $t_{j-1}$ and $t_j$ are consecutive in $A$, then $(t_{j-1},t_j)\in\mci(A)$ and we have $\alpha|_{[t_{j-1},t_j]}\simeq \beta|_{[t_{j-1},t_j]}$. In the second case, $(p_{\alpha})|_{[t_{j-1},t_j]}\cdot (p_{\beta})|_{[t_{j-1},t_j]}^{-}$ is a loop in some element of $\scrv$ and thus $f\circ ((p_{\alpha})|_{[t_{j-1},t_j]}\cdot (p_{\beta})|_{[t_{j-1},t_j]}^{-})=\alpha|_{[t_{j-1},t_j]}\cdot\beta|_{[t_{j-1},t_j]}^{-}$ is null-homotopic in $X$. Since $\alpha|_{[t_{j-1},t_j]}\simeq \beta|_{[t_{j-1},t_j]}$ for all $j\in\{1,2,\dots ,m\}$, we have $\alpha|_{[a,b]}\simeq\beta|_{[a,b]}$.
\end{proof}

\begin{remark}
If we replace $\awild(X)$ with $\wild(X)$ in Lemma \ref{awildsetreductiongoodlemma}, a simpler proof is possible: it follows from Lemma \ref{wildisclosedlemma} that if $A$ is a homotopy cut-set for $\alpha$ and $\beta$, then $\alpha^{-1}(\wild(X))$ is closed in $\ui$. Combining Lemma \ref{nonwildlemma1} and Lemma \ref{scatteruplemma} with the fact that $\awild(X)\subseteq \wild(X)$, it is straightforward to check that $B=(A\cap \alpha^{-1}(\wild(X)))\cup\{0,1\}$ is a homotopy cut-set for $\alpha$ and $\beta$ for which $B\cap (0,1)$ is mapped into $\wild(X)$. Unfortunately, we cannot apply this idea in the more general situation since Example \ref{example2} shows that $\alpha|_{A}^{-1}(\awild(X))$ need not be closed.
\end{remark}

\subsection{Reduction to perfect cut-sets}

Recall that a subset $B\subseteq X$ is \textit{perfect (in $X$)} if $B$ is closed in $X$ and $B$ has no isolated points. Every nowhere-dense perfect subset of $\ui$ is homeomorphic to a Cantor set.

\begin{lemma}\label{perfectreductionlemma}
Suppose $X$ has well-defined scattered $\Pi_1$-products. If $A$ is a homotopy cut-set for paths $\alpha,\beta:\ui\to X$, then either $\alpha\simeq\beta$ or there exists a perfect set $P\subseteq A$ such that 
\begin{enumerate}
\item $B=P\cup\{0,1\}$ is a homotopy cut-set for $\alpha$ and $\beta$,
\item if $(c,d)\in\mci(B)$ contains $(a,b)\in\mci(A)$, then $\alpha|_{[c,a]}\simeq \beta|_{[c,a]}$ when $c<a$ and $\alpha|_{[b,d]}\simeq \beta|_{[b,d]}$ when $b<d$.
\end{enumerate}
\end{lemma}

\begin{proof}
If $A$ is scattered, then we have $\alpha\simeq\beta$ since $X$ is assumed to have well-defined scattered $\Pi_1$-products. If $A$ is not scattered, then, by the Cantor-Bendixson Theorem \cite{KechrisDST}, $A$ is the union of a non-empty perfect subset $P$ and a (countable) scattered subset $A\backslash P$. Let $B=P\cup\{0,1\}$ and fix $(c,d)\in\mci(B)$. We apply Lemma \ref{scatteruplemma} to the paths $\alpha|_{[c,d]}$ and $\beta|_{[c,d]}$ to show they are path-homotopic. Let $a,b\in A \cap (c,d)$ are any points such that $a<b$. Since $A\cap[a,b]$ lies in the scattered space $A\backslash P$ and every subspace of a scattered space is scattered, the set $A\cap [a,b]$ is scattered. Thus $A\cap[a,b]$ is a scattered homotopy cut-set for $\alpha|_{[a,b]}$ and $\beta|_{[a,b]}$. Since $X$ is assumed to have well-defined scattered $\Pi_1$-products, we have $\alpha|_{[a,b]}\simeq \beta|_{[a,b]}$. Therefore, Lemma \ref{scatteruplemma} applies to give $\alpha|_{[c,d]}\simeq \alpha|_{[c,d]}$. This proves (1).

For (2), suppose $(a,b)\in\mci(A)$ is contained in $(c,d)\in\mci(B)$. If $c<a$ and $A\cap [c,a]$ is finite, then $\alpha|_{[c,a]}\simeq \beta|_{[c,a]}$ is clear. Otherwise, we apply the argument used in the previous paragraph: given any $s,t\in A$ with $a<s<t<c$, the set $A\cap [s,t]$ is a scattered homotopy cut-set for $\alpha|_{[s,t]}$ and $\beta|_{[s,t]}$ and thus $\alpha|_{[s,t]}\simeq\beta|_{[s,t]}$. Lemma \ref{scatteruplemma} applies to give $\alpha|_{[c,a]}\simeq \beta|_{[c,a]}$. The symmetric argument may be used to show that $\alpha|_{[b,d]}\simeq \beta|_{[b,d]}$ when $b<d$.
\end{proof}

\subsection{Reduction at isolated image points}

\begin{lemma}\label{inductivesteplemma}
Suppose $X$ has well-defined transfinite $\pi_1$-products and $A$ is a homotopy cut-set for $\alpha,\beta:\ui\to X$. If $x$ is an isolated point of $\alpha(A)$ and $\alpha(A)\backslash \{x\}\neq \emptyset$, then there exists a homotopy cut-set $B$ for $\alpha$ and $\beta$ such that
\begin{enumerate}
\item $B\subseteq A$,
\item $\alpha(B\cap (0,1))\subseteq \alpha(A\cap (0,1))\backslash \{x\}$,
\item if $(c,d)\in \mci(B)$ is the component containing $(a,b)\in\mci(A)$, then $\alpha|_{[c,a]}\simeq \beta|_{[c,a]}$ when $c<a$ and $\alpha|_{[b,d]}\simeq \beta|_{[b,d]}$ when $b<d$.
\end{enumerate}
\end{lemma}

\begin{proof}
Define $B=A\backslash \alpha|_{A}^{-1}(x)\cup \{0,1\}$. Since $\{x\}$ is open in $\alpha(A)$, it's clear that $B$ is closed in $A$ and that $B$ is nowhere dense. Conditions (1) and (2) are clearly satisfied under this definition. We must closely analyze the situation to verify that $B$ is a homotopy cut-set for $\alpha$ and $\beta$ and that Condition (3) holds. 

Call a component $(a,b)\in \mci(A)$ a \textit{transition component} if either $\alpha(a)=x$ and $\alpha(b)\neq x$ or $\alpha(a)\neq x$ and $\alpha(b)=x$. We think of these components as describing the segments of the path $\alpha$ that travel between $\{x\}$ and $\alpha(A)\backslash \{x\}$. First, we must show that at least one transition component exists. Suppose otherwise. Then for all $(a,b)\in\mci(A)$, either $\alpha(\{a,b\})=\{x\}$ or $\alpha(\{a,b\})\subseteq \alpha(A)\backslash \{x\}$. Since $x\in\alpha(A)$ and $\alpha(A)\backslash \{x\}\neq \emptyset$, we are now guaranteed intervals $(p,q),(s,t)\in\mci(A)$ such that $\alpha(\{p,q\})=x$ and $\alpha(\{s,t\})\subseteq \alpha(A)\backslash \{x\}$. Using reverse path symmetry, if necessary, we may assume $q<s$. Let $V=\sup\{v \in A\cap [0,s)\mid (u,v)\in\mci(A),\alpha(\{u,v\})=x\}$. Since $A$ is closed, we have $V\in A$ and the continuity of $\alpha$ gives $\alpha(V)=x$. Since we have assumed no transition components exist, the definition of $V$ guarantees that $V$ cannot be the left endpoint of an element of $\mci(A)$. Therefore, we must have an infinite sequence $V<\cdots s_3<t_3<s_2<t_2<s_1<t_1<s<t$ converging to $V$ where $(s_n,t_n)\in\mci(A)$ and $\alpha(\{s_n,t_n\})\subseteq \alpha(A)\backslash\{x\}$. The continuity of $\alpha$ implies that $\{\alpha(s_n)\}\to \alpha(V)$ in $X$ and thus in the subspace $\alpha(A)$. However, since $\alpha(A)\backslash\{x\}$ is closed in $\alpha(A)$, we have $\alpha(V)\in\alpha(A)\backslash\{x\}$; a contradiction. We conclude that a transition component exists.

Since $x$ is a not a limit point of $\alpha(A)$, there can only be finitely many transition components in $\mci(A)$. Enumerate the transition components of $\mci(A)$ as $(a_1,b_1),(a_2,b_2),\dots, (a_m,b_m)$. Let $A_0=A\cap[0,a_1]$, $A_j=A\cap[b_j,a_{j+1}]$, and $A_m=A\cap[b_m,1]$ where it is possible that any given $A_j$ is degenerate, i.e equal to a point. For each $j\in \{0,1,2,\dots,m\}$, either $\alpha(A_j)=x$ or $\alpha(A_j)\subseteq \alpha(A)\backslash \{x\}$ (for otherwise, the second paragraph would imply the existence of another transition component). Let $P=\{j\mid \alpha(A_j)=x\}$ and $Q=\{j\mid \alpha(A_j)\subseteq \alpha(A)\backslash \{x\}\}$ and note that the elements in $P$ and $Q$ must alternate within $\{1,2,\dots,m\}$.

To check that $B$ is a homotopy cut-set for $\alpha$ and $\beta$, we consider the possible ways in which a component $(c,d)\in\mci(B)$ may be constructed by deleting elements $t\in \alpha|_{A}^{-1}(x)\cap (0,1)$. Note that such $t$ will only be deleted if they lie in $A_j$ for $j\in P$ and are neither $0$ nor $1$.
\begin{enumerate}
\item[] Case I: $(c,d)$ could result from joining two consecutive transition components $(a_j,b_j)$ and $(a_{j+1},b_{j+1})$ where $A_j=\{b_j\}=\{a_{j+1}\}$ is degenerate and $\alpha(A_j)=x$. In this situation, we have $\alpha(\{a_j,b_{j+1}\})\subseteq \alpha(A)\backslash \{x\}$ and we get a component $(c,d)=(a_j,b_{j+1})\in\mci(B)$. Since $A$ is a homotopy cut-set for $\alpha$ and $\beta$, we have $\alpha|_{[a_j,b_j]}\simeq \beta|_{[a_j,b_j]}$ and $\alpha|_{[a_{j+1},b_{j+1}]}\simeq \beta|_{[a_{j+1},b_{j+1}]}$. Thus  $\alpha|_{[c,d]}\equiv\alpha|_{[a_j,b_j]}\cdot\alpha|_{[a_{j+1},b_{j+1}]}\simeq \beta|_{[a_{j},b_{j}]}\cdot \beta|_{[a_{j+1},b_{j+1}]}\equiv\beta|_{[c,d]}$. Observe that Condition (3) holds in this case.
\item[] Case II: $(c,d)$ could result from combining $(a_j,b_j)$, $[b_j,a_{j+1}]$, and $(a_{j+1},b_{j+1})$ into a single interval if $b_j<a_{j+1}$ and $\alpha(A_j)=x$. In this situation, $\alpha(\{a_j,b_{j+1}\})\subseteq \alpha(A)\backslash \{x\}$ and we get a component $(c,d)=(a_j,b_{j+1})\in\mci(B)$. As in the previous case, we have $\alpha|_{[a_j,b_j]}\simeq \beta|_{[a_j,b_j]}$ and $\alpha|_{[a_{j+1},b_{j+1}]}\simeq \beta|_{[a_{j+1},b_{j+1}]}$. Since $A_j$ is a homotopy cut-set for loops $\alpha|_{[b_j,a_{j+1}]}$ and $\beta|_{[b_j,a_{j+1}]}$ based at $x$ and $X$ is assumed to have well-defined transfinite $\pi_1$-products, we have $\alpha|_{[b_j,a_{j+1}]}\simeq\beta|_{[b_j,a_{j+1}]}$. Concatenation gives $\alpha|_{[c,d]}\simeq\beta|_{[c,d]}$. Using these homotopies, Condition (3) is obvious in the case when $(a,b)$ is one of $(a_j,b_j)$ or $(a_{j+1},b_{j+1})$. Otherwise, suppose $(a,b)\in\mci(A)$ lies in $[b_j,a_{j+1}]$. Then $A\cap[b_j,a]$ is a homotopy cut-set for the (possibly degenerate) loops $\alpha|_{[b_j,a]}$ and $\beta|_{[b_j,a]}$ based at $x$. Similarly $A\cap[b,a_{j+1}]$ is a homotopy cut-set for (possibly degenerate) loops $\alpha|_{[b_j,a]}$ and $\beta|_{[b_j,a]}$ based at $x$. Using the assumption that $X$ has well-defined transfinite $\pi_1$-products, we have $\alpha|_{[b_j,a]}\simeq\beta|_{[b_j,a]}$ and $\alpha|_{[b,a_{j+1}]}\simeq\beta|_{[b,a_{j+1}]}$. Since $[c,a]=[a_j,b_j]\cup[b_j,a]$ and $[b,d]=[b,a_{j+1}]\cup[a_{j+1},b_{j+1}]$, combining these two path-homotopies with those on $[a_j,b_j]$ and $[a_{j+1},b_{j+1}]$ shows that $\alpha|_{[c,a]}\simeq\beta|_{[c,a]}$ and $\alpha|_{[b,d]}\simeq\beta|_{[b,d]}$. Thus Condition (3) holds.
\item[] Case III: There are also four possible cases at the endpoints where $\alpha(A_0)=x$ or $\alpha(A_m)=x$. If $A_0$ (resp. $A_m$) is degenerate, then no deletion occurs. If $A_0$ is not degenerate, $A_0$ is a homotopy cut-set for the loops $\alpha|_{[0,a_1]}$ and $\beta|_{[0,a_1]}$ based at $x$. Applying a one-sided version of the argument in Case II results in an interval $(c,d)=(0,b_1)\in\mci(B)$. Note that $0\in B$ despite $\alpha(0)=x$. Condition (3) also holds by the one-sided analogue of the argument used in Case II. If $A_m$ is not degenerate, the symmetric argument applies.
\end{enumerate}
Since $B$ is obtained from $A$ by removing the points of $\alpha|_{A}^{-1}(x)\cap (0,1)$ that lie in an alternating finite sequence of the sets $A_j$, a component $(c,d)\in\mci(B)$ can only result from the above cases. We conclude that $B$ is a homotopy cut-set for $\alpha$ and $\beta$ satisfying Conditions (1)-(3).
\end{proof}

\subsection{Transfinite sequences of reductions}

\begin{lemma}\label{limitordinallemma}
Suppose $X$ has well-defined scattered $\Pi_1$-products, $\kappa$ is a countable limit ordinal and $\{A_{\lambda}\}_{\lambda<\kappa}$ is a transfinite sequence of homotopy cut-sets for paths $\alpha,\beta:\ui\to X$ such that 
\begin{enumerate}
\item for any $\lambda<\kappa$, we have $A_{\lambda+1}\subseteq A_{\lambda}$,
\item if $\lambda<\kappa$ is a limit ordinal, then $A_{\lambda}=\bigcap_{\mu<\lambda}A_{\mu}$,
\item if $(c,d)\in\mci(A_{\lambda+1})$ contains $(a,b)\in\mci(A_{\lambda})$, then $\alpha|_{[c,a]}\simeq \beta|_{[c,a]}$ when $c<a$ and $\alpha|_{[b,d]}\simeq\beta|_{[b,d]}$ when $b<d$.
\end{enumerate}
Then $A_{\kappa}=\bigcap_{\lambda<\kappa}A_{\lambda}$ is a homotopy cut-set for $\alpha$ and $\beta$.
\end{lemma}

\begin{proof}
Clearly, $A_{\kappa}$ is closed and nowhere dense. Let $(a_{\kappa},b_{\kappa})$ be a component of $\ui\backslash A_{\kappa}$. We must show that $\alpha|_{[a_{\kappa},b_{\kappa}]}$ is path-homotopic to $\beta|_{[a_{\kappa},b_{\kappa}]}$. Fix any $(a_0,b_0)\in\mci(A_0)$ such that $(a_0,b_0)\subseteq (a_{\kappa},b_{\kappa})$. If $\lambda\leq\kappa$, let $(a_{\lambda},b_{\lambda})\in \mci(A_{\lambda})$ be the unique interval such that $(a_0,b_0)\subseteq (a_{\lambda},b_{\lambda})\subseteq (a_{\kappa},b_{\kappa})$. Note that the transfinite sequences $\{a_{\lambda}\}_{\lambda<\kappa}$ and $\{b_{\lambda}\}_{\lambda<\kappa}$ are monotone but need not be strictly monotone.

We take a moment to verify that, for each limit ordinal $\lambda\leq \kappa$, we have $\bigcup_{\mu<\lambda}(a_{\mu},b_{\mu})=(a_{\lambda},b_{\lambda})$. If we set $\bigcup_{\mu<\lambda}(a_{\mu},b_{\mu})=(s,t)$, then $(s,t)\subseteq (a_{\lambda},b_{\lambda})$ is clear. If $t<b_{\lambda}$, then $t\notin A_{\lambda}$ since $(a_{\lambda},b_{\lambda})$ is a component of $\ui\backslash A_{\lambda}$. Thus $t\notin A_{\mu_0}$ for some $\mu_0<\lambda$. Let $(c_{\mu_0},d_{\mu_0})\in\mci(A_{\mu_0})$ be the component containing $t$ and note $a_{\lambda}<c_{\mu_0}<t<d_{\mu_0}<b_{\lambda}$. Since $\{b_{\mu}\}_{\mu<\lambda}$ is non-decreasing and converges to $t$, there exists, $\mu_1>\mu_0$ such that $c_{\mu_0}<b_{\mu_1}<t<d_{\mu_0}$. However, this is impossible since we cannot have $b_{\mu_1}\in A_{\mu_1}\subseteq A_{\mu_0}$ and $b_{\mu_1}\in (c_{\mu_0},d_{\mu_0})$. We conclude that $t=b_{\lambda}$. The symmetric argument shows $s=a_{\lambda}$.

The previous paragraph directly implies that the non-increasing function $\kappa+1\to \ui$, given by $\lambda\mapsto a_{\lambda}$, is continuous. The image $\{a_{\lambda}\mid \lambda\leq \kappa\}$ is countable and compact and thus a scattered subset of $\ui$. Similarly, $\lambda\mapsto b_{\lambda}$ is non-decreasing and continuous and the set $\{b_{\lambda}\mid \lambda\leq \kappa\}$ is scattered. The union $B=\{a_{\lambda}\mid \lambda\leq \kappa\}\cup \{b_{\lambda}\mid \lambda\leq\kappa\}$ must be scattered as well.

Every component of $[a_{\kappa},b_{\kappa}]\backslash B$, which is not $(a_0,b_0)$ must be of the form $(a_{\lambda+1},a_{\lambda})$ or $(b_{\lambda},b_{\lambda+1})$ for some $\lambda<\kappa$. By Condition (3) in the statement of the lemma, we have $\alpha|_{[a_{\lambda+1},a_{\lambda}]}\simeq \beta|_{[a_{\lambda+1},a_{\lambda}]}$ if $a_{\lambda+1}<a_{\lambda}$ and $\alpha|_{[b_{\lambda},b_{\lambda+1}]}\simeq \beta|_{[b_{\lambda},b_{\lambda+1}]}$ if $b_{\lambda}<b_{\lambda+1}$. Therefore, $B$ is a scattered homotopy cut-set for $\alpha|_{[a_{\kappa},b_{\kappa}]}$ and $\beta|_{[a_{\kappa},b_{\kappa}]}$. Since $X$ is assumed to have well-defined scattered $\Pi_1$-products, we conclude that $\alpha|_{[a_{\kappa},b_{\kappa}]}\simeq\beta|_{[a_{\kappa},b_{\kappa}]}$
\end{proof}

\begin{theorem}\label{bigtheorem}
Suppose $X$ has well-defined transfinite $\pi_1$-products. If $A$ is a homotopy cut-set for paths $\alpha,\beta:\ui\to X$, then either $\alpha\simeq\beta$ or there exists a non-empty perfect set $P\subseteq A$ such that 
\begin{enumerate}
\item $B=P\cup\{0,1\}$ is a homotopy cut-set for $\alpha$ and $\beta$,
\item $\{t\in B\cap (0,1)\mid \alpha(t)\text{ is isolated in }\alpha(B)\}=\emptyset$.
\end{enumerate} 
\end{theorem}

\begin{proof}
Let $A_0=B_0=A$, $P_0=\emptyset$, and set $Y_0=\alpha(B_0)$. We will construct two descending transfinite sequences $\{A_{\lambda}\}$ and $\{B_{\lambda}\}$ of homotopy cut-sets for $\alpha$ and $\beta$ that will nest in an alternating pattern. We will also define $B_{\lambda}=P_{\lambda}\cup\{0,1\}$ for a perfect set $P_{\lambda}$ and $Y_{\lambda}=\alpha(B_{\lambda})=\beta(B_{\lambda})$.

Suppose $A_{\lambda}$, $B_{\lambda}$, $P_{\lambda}$, and $Y_{\lambda}$ have been defined. If $$\{t\in B_{\lambda}\cap (0,1)\mid \alpha(t)\text{ is isolated in }\alpha(B_{\lambda})\}=\emptyset$$ or if $|\alpha(B_{\lambda})|=1$, let $A_{\lambda+1}=B_{\lambda+1}=B_{\lambda}$. Otherwise, find $t\in B_{\lambda}$ such that $\alpha(t)=y_{\lambda+1}$ is an isolated point of $Y_{\lambda}$. Since $|\alpha(B_{\lambda})|>1$, we have $\alpha(B_{\lambda})\backslash\{y_{\lambda+1}\}\neq\emptyset$ and we may apply Lemma \ref{inductivesteplemma}, to find a homotopy cut-set $A_{\lambda+1}\subseteq B_{\lambda}$ such that 
\begin{itemize}
\item $\alpha(A_{\lambda+1}\cap (0,1))\subseteq \alpha(B_{\lambda}\cap (0,1))\backslash \{y_{\lambda+1}\}$,
\item if $(a,b)\subseteq (c,d)$ for $(a,b)\in\mci(B_{\lambda})$ and $(c,d)\in \mci(A_{\lambda+1})$, then $\alpha|_{[c,a]}\simeq \beta|_{[c,a]}$ and $\alpha|_{[b,d]}\simeq \beta|_{[b,d]}$.
\end{itemize}
Next, by Lemma \ref{perfectreductionlemma}, we may find a perfect set $P_{\lambda+1}\subseteq A_{\lambda+1}$ such that 
\begin{itemize}
\item $B_{\lambda+1}=P_{\lambda+1}\cup\{0,1\}$ is a homotopy cut set for $\alpha$ and $\beta$,
\item if $(a,b)\subseteq (c,d)$ for $(a,b)\in\mci(A_{\lambda+1})$ and $(c,d)\in \mci(B_{\lambda+1})$, then $\alpha|_{[c,a]}\simeq \beta|_{[c,a]}$ and $\alpha|_{[b,d]}\simeq \beta|_{[b,d]}$.
\end{itemize}
Finally, set $Y_{\lambda+1}=\alpha(B_{\lambda+1})$. 

Before dealing with the limit ordinal case, note that we have constructed a transfinite sequences $\{A_{\lambda}\}$ and $\{B_{\lambda}\}$ satisfying $B_{\lambda+1}\subseteq A_{\lambda+1}\subseteq B_{\lambda}\subseteq A_{\lambda}$ for all ordinals $\lambda$. By combining the second bullet points in our choice of $A_{\lambda+1}$ and $B_{\lambda+1}$, it follows that if $(c',d')\subseteq (a,b)\subseteq (c,d)$ for $(c,d)\in\mci(B_{\lambda+1})$, $(a,b)\in\mci(A_{\lambda+1})$, and $(c',d')\in\mci(B_{\lambda})$, then $\alpha|_{[d',d]}\simeq\beta|_{[d',d]}$ when $d'<d$ and $\alpha|_{[c,c']}\simeq\beta|_{[c,c']}$ when $c<c'$.

If $\lambda$ is a limit ordinal, we set $A_{\lambda}=\bigcap_{\mu<\lambda}A_{\mu}$ and $B_{\lambda}=\bigcap_{\mu<\lambda}B_{\mu}$ (noting that $A_{\lambda}=B_{\lambda}$), and $Y_{\lambda}=\alpha(A_{\lambda})=\alpha(B_{\lambda})$. With this definition, $\{A_{\lambda}\}$ is a descending transfinite sequence of closed subsets in the second countable space $A$ and, therefore, must stabilize at a countable ordinal, i.e. there exists countable $\kappa$ such that $A_{\lambda}=A_{\kappa}$ for $\lambda\geq \kappa$ \cite{KechrisDST}. In particular, $A_{\kappa}\supseteq B_{\kappa}\supseteq A_{\kappa+1}$ and thus $A_{\kappa}=B_{\kappa}$. The previous paragraph allows us to apply Lemma \ref{limitordinallemma} to $\{B_{\mu}\}_{\mu<\lambda}$ for any countable limit ordinal $\lambda$. It follows that for any countable limit ordinal $\lambda$ (and thus for an arbitrary limit ordinal), $A_{\lambda}=B_{\lambda}$ is a homotopy cut-set for $\alpha$ and $\beta$.

We have already argued that $B=A_{\kappa}=B_{\kappa}$ is a homotopy cut-set for $\alpha$ and $\beta$. The first possibility is that $|\alpha(B)|=1$ in which case, $\alpha,\beta$ are loops. Since $X$ is assumed to have well-defined transfinite $\pi_1$-products, we have $\alpha\simeq\beta$. The second possibility is that $|\alpha(B)|>1$. If this occurs, then there can be no point $t\in B\cap (0,1)$ such that $\alpha(t)$ is an isolated point of $\alpha(B)$ since our construction would force $A_{\kappa+1}$ to be a proper subset of $B_{\kappa}$. Finally, if $P_{\kappa}=\emptyset$, we have $B=\{0,1\}$ and thus $\alpha\simeq\beta$. Otherwise, $B=P_{\kappa}\cup \{0,1\}$ for the non-empty perfect set $P_{\kappa}$.
\end{proof}

\section{A Proof of Theorem \ref{mainthm}}\label{finalsection}

The following theorem is the culmination of the sequence of lemmas in the previous section. Recall that a space $W$ is \textit{dense-in-itself} if $W$ has no isolated points. We make a point to distinguish a space $W$ being dense-in-itself and $W$ being a perfect subset of a space $X$ since the space $\awild(X)$ need not be closed in $X$ and since we are not assuming that $X$ satisfies any separation axioms.

\begin{theorem}\label{maintechthm}
Suppose $X$ has well-defined transfinite $\pi_1$-products and $x_0\notin \awild(X)$ is a point such that $\{x_0\}$ is closed in $X$. If $A$ is a homotopy cut-set for loops $\alpha$ and $\beta$ based at $x_0$, then either $\alpha\simeq\beta$ or there exists a homotopy cut-set $B\subseteq A$ for $\alpha$ and $\beta$ such that
\begin{enumerate}
\item $B\backslash\{0,1\}$ is non-empty and perfect,
\item $\alpha(B\backslash \{0,1\})=\beta(B\backslash \{0,1\})\subseteq \awild(X)$,
\item $\alpha(B\backslash \{0,1\})=\beta(B\backslash \{0,1\})$ is dense-in-itself.
\end{enumerate}
\end{theorem}

\begin{proof}
Since $x_0\notin\awild(X)$, Lemma \ref{isolatedendpoints} allows us to find a homotopy cut-set $B_1\subseteq A$ for $\alpha$ and $\beta$ such that $0$ and $1$ are isolated in $B_1$, i.e. $B_1\backslash \{0,1\}$ is closed. By Lemma \ref{awildsetreductiongoodlemma}, there is a homotopy cut-set $B_2\subseteq B_1$ for $\alpha$ and $\beta$ such that $\alpha(B_2\backslash\{0,1\})=\beta(B_2\backslash\{0,1\})\subseteq \awild(X)$. By Theorem \ref{bigtheorem}, either $\alpha\simeq\beta$ or there exists a non-empty perfect set $P\subseteq B_2$ such that $B_3=P\cup \{0,1\}$ is a homotopy cut-set for $\alpha$ and $\beta$ and such that $\{t\in B_3\backslash\{0,1\}\mid \alpha(t)\text{ is isolated in }\alpha(B_3)\}=\emptyset$. Since $B_3\subseteq B_1$, both $0$ and $1$ are isolated points in $B_3$ and $P\subseteq (0,1)$. Thus $B_3\backslash \{0,1\}=P$ is perfect. Since $B_3\subseteq B_2$, we have $\alpha(B_3\backslash\{0,1\})=\beta(B_3\backslash\{0,1\})\subseteq \awild(X)$. 

Finally, suppose, to obtain a contradiction, that we have a point $t\in P$ such that $\alpha(t)$ is isolated in $\alpha(P)$. By our choice of $B_3$, the point $\alpha(t)$ is not isolated in $\alpha(B_3)$. Since $\{\alpha(t)\}$ is open in $\alpha(P)$ but not in $\alpha(B_3)=\alpha(P)\cup\{x_0\}$, we have $x_0\notin \alpha(P)$. Write $\{\alpha(t)\}=V\cap \alpha(P)$ for an open set $V$ in $\alpha(B_3)$. This forces $V=\{x_0,\alpha(t)\}$. However, since $\{x_0\}$ is closed in $X$, the complement $\alpha(B_3)\backslash \{x_0\}$ is open in $\alpha(B_3)$. Thus $V\cap (\alpha(B_3)\backslash \{x_0\})=\{\alpha(t)\}$ is open in $\alpha(B_3)$; a contradiction. We conclude that $B=B_3$ is the desired homotopy cut-set.
\end{proof}

Assuming the well-definedness of transfinite $\pi_1$-products, we consider Theorem \ref{maintechthm} to be an optimal reduction result. This claim is justified in part by the example used to prove \cite[Theorem 1.1]{BrazDense}, which is a locally path-connected metric space $X$ for which $\awild(X)$ is perfect (namely, a Cantor set) and for which $X$ has well-defined transfinite $\pi_1$-products but does not have well-defined transfinite $\Pi_1$-products. Prior to our proof of Theorem \ref{mainthm}, we provide a brief summary.

\begin{remark}
Recall that a cut-set $A$ for paths $\alpha,\beta:\ui\to X$ corresponds to path-factorizations $\prod_{I\in \mci(A)}\alpha|_{\ov{I}}$ and $\prod_{I\in \mci(A)}\beta|_{\ov{I}}$ over the cut-map $\alpha|_{A}=\beta|_{A}:A\to X$. If $A$ is a homotopy cut-set, then $[\alpha|_{\ov{I}}]=[\beta|_{\ov{I}}]$ for all $I\in\mci(A)$. Many of the points $t\in A\cap(0,1)$ may be deleted to form the homotopy cut-set $B\subseteq A$ for $\alpha$ and $\beta$ described in Theorem \ref{maintechthm}. The deletion of $t$ may be possible due to one of the following reasons:
\begin{enumerate}
\item $t$ lies in the maximal scattered subset of $A$, in which case $t$ can be removed as part of the reduction process that inductively combines consecutive elements of $\mci(A)$. This is possible since, by Theorem \ref{scatteredproducthm}, $X$ necessarily has well-defined scattered $\Pi_1$-products.
\item $\alpha(t)\notin \awild(X)$ in which case $t$ does not contribute to algebraic wildness,
\item $\alpha(t)$ lies in the maximal scattered subset of $\awild(X)$, in which case $t$ may be removed as part of an inductive process that applies the well-definedness of transfinite $\pi_1$-products at these points.
\end{enumerate}
Interestingly, the author was not able prove Theorem \ref{maintechthm} by invoking the transfinite reductions described in (1) and (3) separately, hence, the alternating transfinite induction used in the proof of Theorem \ref{bigtheorem}. Theorem \ref{maintechthm} allows us to make the reduction from $A$ to $B$ in one step. Once we have $B=P\cup\{0,1\}$ for perfect set $P$ and if $a=\min(P)$, $b=\max(P)$, then we may write $\alpha\equiv \alpha|_{[0,a]}\cdot\left( \prod_{J\in \mci(P)}\alpha|_{\ov{J}}\right)\cdot \alpha|_{[b,1]}$ and $\beta\equiv \beta|_{[0,a]}\cdot\left( \prod_{J\in \mci(P)}\beta|_{\ov{J}}\right)\cdot \beta|_{[b,1]}$ where the first and last factors may or may not be degenerate depending on whether or not $0,1$ lie in $P$. Moreover, $[\alpha|_{\ov{J}}]=[\beta|_{\ov{J}}]$ for all $J\in\mci(B)$ and the image of $\alpha|_{P}=\beta|_{P}$ is a dense-in-itself (and perfect if $X$ is Hausdorff) subspace of $\awild(X)$.
\end{remark}

\begin{proof}[Proof of Theorem \ref{mainthm}]
Since loops are paths, the well-definedness of transfinite $\Pi_1$-products implies the well-definedness of $\pi_1$-products for all spaces. Let $X$ be a space with well-defined $\pi_1$-products and such that $\awild(X)$ is scattered. Fix a point $a_0\in X$. By restricting to the path component of $a_0$, we may assume that $X$ is path connected. By Remark \ref{homotopyinvarianceremark}, we may replace $X$ with the homotopy equivalent space $X\cup [0,1]/\sim$ obtained by attaching a copy of $\ui$ to $X$ by identifying $0\sim a_0$ and taking the image of $1$ to be the new basepoint $x_0$. Note that $\{x_0\}$ is closed in $X$ and that we have $x_0\notin\awild(X)$ since $X$ is locally contractible at $x_0$.

Suppose $A$ is a homotopy cut-set for paths $\alpha,\beta:\ui\to X$. Let $\gamma_1$ and $\gamma_2$ be paths in $X$ from $x_0$ to $\alpha(0)=\beta(0)$ and $\alpha(1)=\beta(1)$ respectively. Then $\{0,1\}\cup \{\frac{t+1}{3}\mid t\in A\}$ is a homotopy cut-set for the loops $\gamma_1\cdot\alpha\cdot\gamma_{2}^{-}$ and $\gamma_1\cdot\beta\cdot\gamma_{2}^{-}$ based at $x_0$. Replacing $\alpha$ and $\beta$ with these loops and $A$ with the corresponding homotopy-cut set allows us to assume from the start that $\alpha$ and $\beta$ are loops based at $x_0$. We may now apply Theorem \ref{maintechthm} from which we have that either $\alpha\simeq\beta$ or there exists a homotopy cut-set $B\subseteq A$ for $\alpha$ and $\beta$ such that $\alpha(B\backslash \{0,1\})=\beta(B\backslash \{0,1\})$ is a dense-in-itself set contained in the scattered set $\awild(X)$. However, the only dense-in-itself subspace of a scattered space is the empty set. Thus, $B\backslash \{0,1\}=\emptyset$. Since $B=\{0,1\}$ is a homotopy cut-set for $\alpha$ and $\beta$, we conclude that $\alpha\simeq\beta$.
\end{proof}


\begin{thebibliography}{99}
\expandafter\ifx\csname url\endcsname\relax

\fi
\expandafter\ifx\csname urlprefix\endcsname\relax

\fi

\bibitem{BrazScattered} J.~Brazas, \emph{Scattered products in fundamental groupoids}, Proc. Amer. Math. Soc. {\bf 148} (2020), no. 6, 2655--2670.

\bibitem{BrazDense} J.~Brazas, \emph{Dense products in fundamental groupoids}, J. Homotopy and Related Structures {\bf 14} (2019) 1083--1102.

\bibitem{BFTestMap} J.~Brazas, H.~Fischer, \emph{Test map characterizations of local properties of fundamental groups}, J. Topology and Analysis {\bf 12} (2020) 37-85.

\bibitem{BFThawaiianpants} J.~Brazas, H.~Fischer, \emph{On the failure of the first \v Cech homotopy group to register geometrically relevant fundamental group elements}, Preprint. 2019. arXiv:1902.08887

\bibitem{CConedim} J.W.~Cannon, G.R.~Conner, \textit{On the fundamental groups of one-dimensional spaces}, Topology Appl. {\bf 153} (2006) 2648--2672.

\bibitem{CChegroup} J.W.~Cannon, G.R.~Conner, \textit{The combinatorial structure of the Hawaiian earring group}, Topology Appl. {\bf 106} (2000) 225--271.

\bibitem{CSWDiffeology} J.D.~Christensen, G.~Sinnamon, E.~Wu, \textit{The D-topology for diffeological spaces}, Pacific J. Math. {\bf 272} (2014), no. 1, 87--110.

\bibitem{CHM} G.R.~Conner, W.~Hojka, M.~Meilstrup, \textit{Archipelago groups}, Proc. Amer. Math. Soc. {\bf 143} (2015) 4973--4988.

\bibitem{CK} G.R.~Conner, C.~Kent, \textit{Fundamental groups of locally connected subsets of the plane}, Advances in Mathematics {\bf 347} (2019) 384--407. 

\bibitem{CMRZZ08} G.R.~Conner, M.~Meilstrup, D.~Repov\u{s}, A.~Zastrow, M.~\u{Z}eljko, \textit{On small homotopies of loops}, Topology Appl. {\bf 155} (2008) 1089--1097.

\bibitem{ConnerSpencer} G.~Conner, K.~Spencer, \emph{Anomalous behavior of the Hawaiian earring group}, J. of Group Theory {\bf 8} (2005) 223--227.

\bibitem{Edafreesigmaproducts} K. Eda, \emph{Free $\sigma$-products and noncommutatively slender groups}, J. Algebra {\bf 148} (1992), 243–-263.

\bibitem{E99freesubgroups} K.~Eda, \emph{Free Subgroups of the Fundamental Group of the Hawaiian Earring}, J. Algebra. {\bf 219} (1999) 598--605.

\bibitem{EdaSpatial} K.~Eda, \emph{The fundamental groups of one-dimensional spaces and spatial homomorphisms}, Topology Appl. {\bf 123} (2002) 479--505.

\bibitem{Edaonedim} K. Eda, \emph{Homotopy types of one-dimensional Peano continua}, Fund. Math. {\bf 209} (2010) 27--42.

\bibitem{Edasingularonedim16} K. Eda, \emph{Singular homology groups of one-dimensional Peano Continua}, Fund. Math. {\bf 232} (2016) 99--115.

\bibitem{EK98} K.~Eda, K.~Kawamura, \emph{The fundamental groups of one-dimensional spaces}, Topology Appl. {\bf 87} (1998) 163--172.

\bibitem{FRdirectedhomotopy} L.~Fajstrup, J.~Rosick\'{y}, \emph{A convenient category for directed homotopy}, Theory Appl. Categ. {\bf 21} (2008) 7–-20.
%
\bibitem{FRVZ11} H.~Fischer, D.~Repov\v{s}, Z.~Virk, and A.~Zastrow, \emph{On semilocally simply-connected spaces}, Topology Appl. {\bf 158} (2011) 397--408.

\bibitem{FZ07} H.~Fischer and A.~Zastrow, \emph{Generalized universal
covering spaces and the shape group}, Fund. Math. {\bf 197} (2007) 167--196.

\bibitem{FZ013caley} H.~Fischer, A.~Zastrow, \emph{Combinatorial $\bbr$-trees as generalized Caley graphs for fundamental groups of one-dimensional spaces}, Geometriae Dedicata {\bf 163} (2013) 19--43.

\bibitem{KechrisDST} A.S.~Kechris, \textit{Classical Descriptive Set Theory},  Springer-Verlag, New York, 1995.

\bibitem{KentHomomorphisms} C.~Kent, {\it Homomorphisms of fundamental groups of planar continua}, Pacific J. Math. {\bf 295} (2018), no. 1, 43-–55.

\bibitem{Meilstrup} M.~Meilstrup, \textit{Classifying homotopy types of one-dimensional Peano continua}, 2005, Masters Thesis, Brigham Young University.

\bibitem{RosensteinLO} J.G.~Rosenstein, \emph{Linear Orderings}, Pure and Applied Mathematics Series, Volume 98. Academic Press, New York, 1982.

\bibitem{Spanier66} E.~Spanier, \textit{Algebraic Topology}, McGraw-Hill, 1966.


\end{thebibliography}
\end{document}